\documentclass[12pt,twoside]{amsart}
\usepackage{amssymb}
\usepackage{verbatim}
\usepackage{amsmath}
\usepackage{bm}
\usepackage{a4wide}
\usepackage[T1]{fontenc}
\usepackage{times}
\usepackage{amssymb,latexsym}
\usepackage{enumerate}
\usepackage{pict2e}

\makeatletter
\DeclareRobustCommand{\intprod}{%
  \mathbin{\mathpalette\int@prod{(0.1,0)(0.9,0)(0.9,0.8)}}%
}
\DeclareRobustCommand{\intprodr}{%
  \mathbin{\mathpalette\int@prod{(0.1,0.8)(0.1,0)(0.9,0)}}}

\newcommand{\int@prod}[2]{%
  \begingroup
  \sbox\z@{$\m@th#1+$}%
  \setlength\unitlength{\wd\z@}%
  \begin{picture}(1,1)
  \roundcap
  \polyline#2
  \end{picture}%
  \endgroup
}
\makeatother

\makeatletter
\newcommand{\sumprime}{\if@display\sideset{}{'}\sum%
            \else\sum'\fi}
\makeatother

\allowdisplaybreaks
\begin{document}

\numberwithin{equation}{section}

\newtheorem{theorem}{Theorem}[section]
\newtheorem{prop}[theorem]{Proposition}
\newtheorem{conjecture}[theorem]{Conjecture}
\def\theconjecture{\unskip}
\newtheorem{corollary}[theorem]{Corollary}
\newtheorem{lemma}[theorem]{Lemma}
\newtheorem{observation}[theorem]{Observation}
\newtheorem{definition}{Definition}
\newtheorem*{definition*}{Definition}
\numberwithin{definition}{section} 
\newtheorem{remark}{Remark}
\newtheorem*{note}{Note}
\def\theremark{\unskip}
\newtheorem{kl}{Key Lemma}
\def\thekl{\unskip}
\newtheorem{question}{Question}
\def\thequestion{\unskip}
\newtheorem*{example}{Example}
\newtheorem{problem}{Problem}

\thanks{}

\title{Some characterizations of weakly uniformly perfect sets}

 \author[Zhiyuan Zheng]{Zhiyuan Zheng}
\date{2025. 07. 16}



\address[Zhiyuan Zheng]{Department of Mathematics and Computer Sciences, Tongling University, Anhui, 244000, China}

\email{2023052@tlu.edu.cn}

\begin{abstract}
In this paper, the concept of weakly uniform perfectness is considered. As an analogue of the theory of uniform perfectness, we establish the relationships between weakly uniform perfectness and Bergman kernel, Poincar\'e metric, local harmonic measure and Hausdorff content. In particular, for a bounded domain $\Omega \subset \mathbb{C}$, we show that the uniform perfectness of $\partial \Omega$ is equivalent to $K_{\Omega}(z) \gtrsim \delta_{\Omega}(z)^{-2}$, where $K_{\Omega}(z)$ is the Bergman kernel of $\Omega$ and $\delta_{\Omega}(z)$ denotes the boundary distance.



\end{abstract}

\maketitle

\section{Introduction}

The concept of uniform perfectness is introduced by Pommerenke in 1979 \cite{Pommerenke}, which found connections with various problems in complex analysis, geometry, dynamic system and other fields (see, e.g.,  \cite{Ancona,Fernandez,Gonzalez,Hinkkanen,Jarvi,Lithner,Rocha,Osgood,Pommerenke,Pommerenke1984,Sugawa1998,Sugawa2003,Tukia} ).  
A closed set $K \subset \mathbb{C}\cup \{\infty\}$ is said to be uniformly perfect if there exists a constant $C>0$ such that 
$$
\{z \in \mathbb{C}; Cr\le |z-a|\le r\}\cap K \ne \emptyset
$$ 
for every $a \in K$ and $r\in (0,\mathrm{diam}(K))$. 

The concept of uniform perfection is known to have numerous equivalent characterizations (see,e.g., \cite{Ancona,Chen,Jarvi,Pommerenke,Sugawa2001,Sugawa2003}). Below, we list a few: the uniform perfectness of a closed set $K$ is equivalent to the following statements:

$(i)$ there exists a constant $C>0$ such that 
$\mathrm{Cap}(K_r(a))\ge Cr$
for every $r\in (0,\mathrm{diam}(K))$ and $a \in K$,  where $K_r(a):=\overline{D(a,r)} \cap K$, and $\mathrm{Cap}(\cdot)$ denotes the logarithm capacity; 

$(ii)$ the Bergman kernel of $\Omega$ satisfies $$K_{\Omega}(z)\gtrsim \delta_{\Omega}(z)^{-2},$$ and the Bergman metric $b_{\Omega}(z)|\mathrm{d}z|$ satisfies $$b_{\Omega}(z)\gtrsim \delta_{\Omega}(z)^{-1},$$ where $\Omega$ is assumed to be a domain in $\mathbb{C}$ with $\partial \Omega=K$, and $\delta_{\Omega}(z)$ denotes the distance of $z$ to $\partial \Omega$;

$(iii)$ the Poincar\'e metric $\rho_{\Omega}(z)|\mathrm{d}z|$ of $\Omega$ satisfies
$\rho_{\Omega}(z) \gtrsim \delta_{\Omega}(z)^{-1}$, where the domain $\Omega$ is assumed as above; 


$(iv)$ $\Omega$ is uniformly $\Delta$-regular, where the domain $\Omega$ is assumed as above;

$(v)$ there exist constants $\alpha, A>0$ such that  for every $a\in K$ and $r \in \left( 0,\frac{\mathrm{diam}(\Omega)}{2}\right)$, we have
$$
\Lambda^{\alpha}(K_r(a)) \ge A\cdot r^{\alpha},
$$
where $\Lambda^{\alpha}(K)$ denotes the $\alpha$-Hausdorff content of $K\subset \mathbb{C}$.

In \cite{XiongZheng}, a type of domains in $\mathbb{C}$ is studied, whose properties are analogous to those of domains with uniformly perfect boundary. In this paper, a closed set $K \subset \mathbb{C}$ is called $h$-uniformly perfect if there exists a constant $r_0>0$ such that
$$
\left\{ z\in \mathbb{C}; h(r) \le |z-a| \le r        \right\} \cap K \ne \emptyset
$$
for every $a \in K$ and $r \in (0,r_0)$, where $h$ is an increasing function on $[0,r_0)$ satisfying $h(r)\le r$. In particular, if $K$ is bounded and $h$-uniformly perfect with $h(r)=cr \,(0<c<1)$, then $K$ is uniformly perfect. 

We hope to explore this concept more deeply; however, due to certain difficulties in addressing the general situations, we focus on some special cases in this paper. Let $h_{1,\alpha}(r)=r^{\alpha}$ and $h_{2,\beta}(r)=r(-\log r)^{-\beta}$ for some $\alpha>1$ and $\beta>0$. Following \cite{XiongZheng}, we say that a set $E \subset \mathbb{C}$ satisfies condition $(U)_{1,\alpha}$ or $(U)_{2,\beta}$ if it is $Ch_{1,\alpha}$-uniformly perfect or $Ch_{2,\beta}$-uniformly perfect for some $C>0$, respectively. 

A generalization of statement (i) has already been obtained in \cite{XiongZheng}. It is also shown in \cite{XiongZheng} that if the boundary of $\Omega \subset \mathbb{C}$ satisfies $(U)_{1,\alpha}$ with $1<\alpha <2$ or $(U)_{2,\beta}$ with $\beta>0$, then we have
$$
K_{\Omega}(z)\gtrsim \frac{1}{\delta_{\Omega}(z)^2 \log \frac{1}{\delta_{\Omega}(z)}},\,\,\,z \to \partial \Omega
$$
or
$$
K_{\Omega}(z) \gtrsim \frac{1}{\delta_{\Omega}(z)^2\log \log \frac{1}{\delta_{\Omega}(z)}},\,\,\, z\to \partial \Omega,
$$
respectively.
Here, we shall prove the following theorem:
\begin{theorem} Let $\Omega$ be a planar domain and $z \to \partial \Omega$.

$(1)$ If $\Omega$ is bounded and $K_{\Omega}(z)\gtrsim \delta_{\Omega}(z)^{-2}$, then $\partial \Omega$ is uniformly perfect.

$(2)$If $K_{\Omega}(z) \ge \frac{M}{\delta_{\Omega}(z)^2\log \frac{1}{\delta_{\Omega}(z)}}$ for some $M>0$, then there exists some $\alpha>1$ such that $\partial \Omega$ satisfies condition $(U)_{1,\alpha}$. In particular, if $M>\frac{1}{2\pi}$, then $\alpha \in (1,2)$.

$(3)$ If $K_{\Omega}(z) \gtrsim \frac{1}{\delta_{\Omega}(z)^2\log \log \frac{1}{\delta_{\Omega}(z)}}$, then there exists $\beta>0$ such that $\partial \Omega$ satisfies condition $(U)_{2,\beta}$.
\end{theorem}

In particular, if $\Omega$ is bounded, the uniform perfectness of $\partial \Omega$ is equivalent to that $K_{\Omega}(z) \asymp \delta_{\Omega}(z)^{-2}$ , which simplifies the statement $(ii)$. However, it remains unknown whether the uniform perfectness of $\partial \Omega$ is also equivalent to $b_{\Omega}(z)\asymp \delta_{\Omega}(z)^{-1}$.

The concept of $h$-uniform perfectness is also related to the boundary behavior of $\rho_{\Omega}(z)$, the density function of Poincar\'e metric.
\begin{theorem} 
Let $\rho_{\Omega}(z)|\mathrm{d}z|$ be the Poincar\'e metric of a planar domain $\Omega$ and $z \to \partial \Omega$.

$(1)$ If $\rho_{\Omega}(z) \ge \frac{M}{\delta_{\Omega}(z)\log \frac{1}{\delta_{\Omega}(z)}}$ for some $M>1$, then  $\partial \Omega$ satisfies $(U)_{1,\alpha}$ for some $\alpha>1$.

$(2)$ $\rho_{\Omega}(z)\gtrsim \frac{1}{\delta_{\Omega}(z)\log \log \frac{1}{\delta_{\Omega}(z)}} \Leftrightarrow \partial \Omega$ satisfies $(U)_{2,\beta}$ for some $\beta >0$.
\end{theorem}
The condition $M>1$ in $(1)$ is sharp, since $\rho_{\mathbb{D}-\{0\}}(z)=\frac{1}{|z|\log \frac{1}{|z|}}$.  The proof of Theorem 1.2 is analogous to that of Theorem 1.1, which is inspired by Beardon-Pommerenke \cite{BeardonPommerenke}. 

The concept of uniformly $\Delta$-regular is introduced by Ancona in \cite{Ancona}. A domain $\Omega$ is said to be $\Delta$-regular, if there exists a constant $\varepsilon>0$ such that
$$
w_{a,r}(z)\le 1-\varepsilon
$$
for every $a\in \partial \Omega,r\in (0,\mathrm{diam}(\Omega))$, and $z\in \partial D(a,r/2)$, where $w_{a,r}$ denotes the harmonic measure of $\Omega \cap \partial D(a,r)$ in the domain $\Omega \cap D(a,r)$. Analogously, we say that a domain $\Omega \subset \mathbb{C}$ satisfies condition $(\Delta)_h$, if there exist $\varepsilon \in (0,1)$ and $r_0>0$ such that 
$$
w_{a,r}(z)\le 1-\varepsilon
$$  
for every $a\in \partial \Omega, r\in (0,r_0)$, and $z\in \Omega \cap \partial D(a,h(r))$. For $\alpha>1$ and $\beta>0$, let $h_{1,\alpha}(t)=t^{\alpha}$, and $h_{2,\beta}(t)=t\left(\log \frac{1}{t}\right)^{-\beta}$. A domain $\Omega$ is said to satisfy $(\Delta)_{1,\alpha}$ or $(\Delta)_{2,\beta}$ if there exists a $C>0$ such that $\Omega$ satisfies condition $(\Delta)_{Ch_{1,\alpha}}$ or $(\Delta)_{Ch_{2,\beta}}$, respectively. We obtain

\begin{theorem}
Let $\Omega$ be a domain in $\mathbb{C}$.

$(1)$ If $\partial \Omega$ satisfies condition $(U)_{1,\alpha}$ for some $\alpha \in (1,2)$, then there exists a constant $\alpha'>0$ such that $\Omega$ satisfies $(\Delta)_{1,\alpha'}$. Conversely, if $\Omega$ satisfies $(\Delta)_{1,\alpha'}$ for some $\alpha'>0$, then there exists a constant $\alpha>1$ such that $\partial \Omega$ satisfies $(U)_{1,\alpha}$.

$(2)$ $\partial \Omega$ satisfies $(U)_{2,\beta}$ for some $\beta>0$ if and only if $\Omega$ satisfies $(\Delta)_{2,\beta'}$ for some $\beta'>0$.
\end{theorem}

Now we focus on the statement (v), which implies that the Hausdorff dimension of a uniformly perfect set is always positive. However, for a closed set $E \subset \mathbb{C}$ only meets condition $(U)_{1,\alpha}$ or $(U)_{2,\beta}$, it may not be detected by the classical Hausdorff dimension (see Appendix). Here we provide an analogue of statement (v):

\begin{theorem}
Let $E$ be a closed set in $\mathbb{C}$.

$(1)$ $E$ satisfies contion $(U)_{1,\alpha}$ for some $\alpha >1 \Leftrightarrow$ there exist positive constants $A,C,r_0$ and $\gamma$ such that
$$
\Lambda_{g_{1,\gamma}}(E\cap \overline{B}(a,r)) \ge A \cdot g_{1,\gamma}(2r)
$$
for every $a \in E$ and $r \in (0,r_0)$, where $g_{1,\gamma}(t)$ is an increasing function on $(0,+\infty)$ such that 
$$g_{1,\gamma}(t)=\left( \log \frac{1}{Ct} \right)^{-\gamma}
$$
on $(0,2r_0),$ and $\Lambda_{g_{1,\gamma}}$ denotes the $g_{1,\gamma}$-Hausdorff content.
 
$(2)$ $E$ satisfies condition $(U)_{2,\beta}$ for some $\beta >0 \Leftrightarrow$ there exist positive constants $A,r_0$ and $\eta$ such that
$$
\Lambda_{g_{2,\eta}}(E\cap \overline{B}(a,r)) \ge A \cdot g_{2,\eta}(2r)
$$
for every $a \in E$ and $r \in (0,r_0)$, where $g_{2,\eta}(t)$ is an increasing function on $(0,+\infty)$ such that
$$
g_{2,\eta}(t)=\exp \left(-\eta\cdot \frac{\log 2/t}{\log \log 4/t} \right)
$$
on $(0,2r_0),$ and $\Lambda_{g_{2,\eta}}$ denotes the $g_{2,\eta}$-Hausdorff content.
\end{theorem}

\section{Proof of Theorem 1.1}
Let $R=\{z \in \mathbb{C}; r<|z|<1\}$. Since $$\left\{\frac{z^n}{\|z^n\|_{L^2(R)}}; n \in \mathbb{Z}\right\}
$$ forms an orthnormal basis of $A^2(R)$,  
 the Bergman kernel of $A^2(R)$ can be represented by
$$
K_{R}(z,w)= \frac{1}{2\pi z \overline{w}\log \frac{1}{r}}+\frac{1}{\pi z \overline{w}}\sum_{n\in \mathbb{Z}-\{0\}}\frac{n(z\overline{w})^n}{1-r^{2n}}.
$$
Let $z\in R$ with $|z|=r^t$, where $r,t \in (0, \frac{1}{2}]$, then
\begin{eqnarray}
\notag K_R(z)&=&\frac{1}{2\pi r^{2t}\log \frac{1}{r}}+\frac{1}{\pi r^{2t}}\left(\sum_{n=1}^{+\infty}\frac{nr^{2nt}}{1-r^{2n}}+\sum_{n=1}^{+\infty}\frac{(-n)r^{-2nt}}{1-r^{-2n}} \right)\\
\notag &\le& \frac{1}{2\pi r^{2t}\log \frac{1}{r}} +\frac{1}{\pi r^{2t}}\sum_{n=1}^{\infty}\frac{n(r^{2nt}+r^{2n-2nt})}{1-r^{2n}}   \\
\notag &\le &\frac{1}{2\pi r^{2t}\log \frac{1}{r}}+\frac{2}{\pi r^{2t}}\sum_{n=1}^{+\infty}\frac{nr^{2nt}}{1-r^{2n}}\\
\notag &\le &\frac{1}{2\pi r^{2t}\log \frac{1}{r}}+\frac{8}{3\pi r^{2t}}\sum_{n=1}^{+\infty} nr^{2nt}\\
\notag &=&\frac{1}{2\pi r^{2t}\log \frac{1}{r}}+\frac{8}{3\pi r^{2t}}\cdot \frac{r^{2t}}{(1-r^{2t})^2}\\
\notag &\le&\frac{1}{2\pi r^{2t}\log \frac{1}{r}}+\frac{8}{3\pi}\cdot \frac{1}{(1-\frac{1}{2^{2t}})^2}\\
&=&\frac{1}{2\pi r^{2t}\log \frac{1}{r}}+\mathrm{const}(t), \label{2.1}
\end{eqnarray}
where $\mathrm{const}(t)$ denotes a constant that depends only on $t$. In particular, if $t=\frac{1}{2}$, then $$K_R(z)\le \frac{1}{2\pi r \log \frac{1}{r}}+\frac{32}{3\pi}.$$

\begin{proof}[Proof of Theorem 1.1]
$(1)$ Notice that $\partial \Omega$ is uniformly perfect is equivalent to the existence of  constants $C\in (0,1)$ and $r_0>0$ such that 
$$
\{ z\in \mathbb{C}; Cr \le |z-a| \le r \} \cap \partial \Omega \ne \emptyset, \,\,\,\forall a\in \partial \Omega, r\in (0,r_0),
$$
since $\partial \Omega$ is bounded. Now we suppose on the contrary that $\partial \Omega$ is not uniformly perfect. Then for every $C \in (0,1)$ and $r_0>0$, there exist $a \in \partial \Omega$ and $r\in (0,r_0)$ such that
$$
\{ z\in \mathbb{C}; Cr \le |z-a| \le r \} \cap \partial \Omega =\emptyset.
$$
We may take $C<\frac{1}{4}, r_0<\frac{\mathrm{diam}(\Omega)}{2}$. Denote $A =\{ z\in \mathbb{C};Cr < |z-a|< r \}$. Notice that $r_0<\mathrm{diam}(\Omega) ,$ if $\overline{A} \subset (\Omega^c)^{o}$, then $\Omega$ is not connected. Therefore, $\overline{A} \subset \Omega$.

Take $z_0 \in A$ with $|z_0-a|=\sqrt{C}\cdot r$. Let
$$
T: A \to R=\{z\in \mathbb{C}; C<|z|<1\}, z \to \frac{z-a}{r}.
$$
Write $z_1=T(z_0)$, then from (\ref{2.1}) we see that
\begin{eqnarray*}
K_{A}(z_0)&=&K_{R}(z_1)|T'(z_0)|^2\\
&\le& \left( \frac{1}{2\pi C\log \frac{1}{C}}+\frac{32}{3\pi} \right)\cdot \frac{1}{r^2}.
\end{eqnarray*}
Since $\delta_{\Omega}(z_0)\le |z_{0}-a|=\sqrt{C}\cdot r$, we have
$$
K_A(z_0) \le \left( \frac{1}{2\pi C \log \frac{1}{C}}+\frac{32}{3\pi} \right)\cdot \frac{C}{\delta_{\Omega}(z_0)^2}.
$$
On the other hand,  since $A$ is a subdomain of $\Omega$, we have $$K_{A}(z_0)\ge K_{\Omega}(z_0) \ge M\cdot \delta_{\Omega}(z_0)^{-2}$$ for some constant $M>0$. 
Thus, for each $C \in (0,\frac{1}{4})$ and $r_0>0$, there exists a point $z_0 \in \Omega$ with $\delta_{\Omega}(z_0)<r_0$ such that
$$
0<M \le C\cdot \left(\frac{1}{2\pi C\log \frac{1}{C}}+\frac{32}{3\pi}\right).
$$
But this leads to a contradiction, since $C$ can be taken sufficiently small. Therefore, $\partial \Omega$ must be uniformly perfect. 

$(2)$ Let $t \in (0,\frac{1}{2}]$. Suppose, the contrary, that  $\partial \Omega$ does not satisfy condition $(U)_{1,\alpha}$ for every $\alpha \in (1,2)$. Then for every $C>0$ and $r_0>0$, there exist $a \in \partial \Omega$ and $r \in (0,r_0)$ such that
$$
\{z\in \mathbb{C}; Cr^{\alpha}\le |z-a| \le r \}\cap \partial \Omega =\emptyset.
$$
Write $A:=\{z\in \mathbb{C}; Cr^{\alpha}< |z-a| < r \}$, then $\overline{A}\subset \Omega$ as in $(1)$. Let
$$
T:A \to R=\left\{ z\in \mathbb{C}; Cr^{\alpha-1} <|z|< 1 \right\}, z \to \frac{z-a}{r}.
$$
Take $z_0\in A$ with $ |z_0-a|=C^tr^{t(\alpha-1)+1}$. Write $z_1=T(z_0)$, then $|z_1|=(Cr^{\alpha-1})^t$. Thus, (\ref{2.1}) implies
\begin{equation}\label{Bergmankernel3}
K_A(z_0)=K_R(z_1)|T'(z_0)|^2 \le \left( \frac{1}{2\pi (Cr^{\alpha-1})^{2t}\log \frac{1}{Cr^{\alpha-1}}}+\mathrm{const}(t) \right) \cdot \frac{1}{r^2}.
\end{equation}
Since
$$
\delta_{\Omega}(z_0) \le |z_0-a| =C^tr^{t(\alpha-1)+1}
$$
and
$$
\delta_{\Omega}(z_0) \ge |z_0-a|-Cr^{\alpha}\ge \frac{C^t}{2}r^{t(\alpha-1)+1}
$$
for  sufficiently small $r_0$, we have
\begin{equation}\label{restimate}
\left( \frac{\delta_{\Omega}(z_0)}{C^t} \right)^{\frac{1}{t(\alpha-1)+1}} \le r \le \left( \frac{2\delta_{\Omega}(z_0)}{C^t} \right)^{\frac{1}{t(\alpha-1)+1}}. 
\end{equation}
Then 
\begin{eqnarray*}
K_{A}(z_0) &\le& \mathrm{const}(t)\left(\frac{C^t}{\delta_{\Omega}(z_0)}\right)^{\frac{2}{t(\alpha-1) +1}} +\frac{1}{2\pi C^{2t} r^{2t(\alpha-1) +2}\left((\alpha -1)\log \frac{1}{r} -\log C \right)}\\
&\le& \mathrm{const}(t,C) \delta_{\Omega}(z_0)^{-\frac{2}{t(\alpha-1) +1}} +\frac{t(\alpha-1)+1}{2\pi(\alpha-1)\delta_{\Omega}(z_0)^2}\cdot\frac{1}{ \left( \log \frac{1}{\delta_{\Omega}(z_0)}+ \mathrm{const}(t,C,\alpha) \right)}
\end{eqnarray*}
in view of  (\ref{Bergmankernel3}) and (\ref{restimate}).
On the other hand, 
$$
K_{A}(z_0) \ge K_{\Omega}(z_0) \ge M \cdot \frac{1}{\delta_{\Omega}(z_0)^2 \log \frac{1}{\delta_{\Omega}(z_0)}},
$$
where $M>0$. Thus, for every $t\in (0,\frac{1}{2}], \alpha \in (1,\infty), C>0$ and $r_0>0$, there exists a point $z_0 \in \Omega$ with $\delta_{\Omega}(z_0) \le r_0$, such that
$$
0<M \le \mathrm{const}(t,C)\delta_{\Omega}(z_0)^{2-\frac{2}{t(\alpha-1)+1}}\log \frac{1}{\delta_{\Omega}(z_0)}+\frac{t(\alpha-1)+1}{2\pi (\alpha-1)}\cdot \frac{\log \frac{1}{\delta_{\Omega}(z_0)}}{\log \frac{1}{\delta_{\Omega}(z_0)}+\mathrm{const}(t,C,\alpha)}. 
$$
Let $C$ be fixed and let $ r_0 \to 0$, then
$$
0<M \le \frac{t(\alpha-1)+1}{2 \pi(\alpha-1)}, \,\,\, \forall \alpha \in (1,+\infty), t\in (0,1/2],
$$
which is a contradiction as $\alpha \to +\infty$ and $ t \to 0$. 

If $M > \frac{1}{2\pi}$, a similar process with $\alpha \to 2$ and $ t \to 0$ implies that $\partial \Omega$ satisfies $(U)_{1,\alpha}$ for some $\alpha \in (1,2)$.

$(3)$ Suppose on the contrary that  $\partial \Omega$ does not satisfies condition $(U)_{2,\beta}$ for every $\beta>0$. Then, for every $C>0$ and $r_0>0$, there exist $a \in \partial \Omega$ and $r\in (0,r_0)$ such that
$$
\left\{ z\in \mathbb{C}; Cr\left(\log \frac{1}{r}\right)^{-\beta} \le |z-a| \le r \right\}\cap \partial \Omega=\emptyset.
$$
Similar to $(1)$, we write $A:=\left\{ z\in \mathbb{C}; Cr\left(\log \frac{1}{r}\right)^{-\beta} < |z-a| < r \right\}$. Then  $\overline{A} \subset \Omega$. Let
$$
T:A \to R=\left\{ z\in \mathbb{C}; C\left(\log \frac{1}{r} \right)^{-\beta} <|z|< 1 \right\}, z \to \frac{z-a}{r}.
$$
Take $z_0 \in A$ with $|z_0-a|=\sqrt{C}r\left(\log \frac{1}{r}\right)^{-\beta/2}.$ Then
\begin{eqnarray}
\notag K_{A}(z_0) &\le& \left( \frac{1}{2\pi C(\log \frac{1}{r})^{-\beta}\log \frac{(\log \frac{1}{r})^{\beta}}{C}} +\frac{32}{3\pi}\right)\cdot \frac{1}{r^2}\\
&=& \left( \frac{1}{2\pi \beta \log\log \frac{1}{r}-2\pi \log C} +\frac{32}{3\pi}C\left(\log \frac{1}{r}\right)^{-\beta}   \right)\cdot \frac{1}{\left(\sqrt{C}r\left(\log \frac{1}{r}\right)^{-\beta /2}\right)^2}. \label{bergmankernelring}
\end{eqnarray}
Here
\begin{equation}\label{deltaless}
\delta_{\Omega}(z_0) \le |z_0-a|=\sqrt{C}r\left(\log \frac{1}{r}\right)^{-\beta /2}
\end{equation} 
and
\begin{equation}\label{deltamore}
\delta_{\Omega}(z_0)\ge |z_0-a|-Cr\left(\log \frac{1}{r}\right)^{-\beta} \ge \frac{\sqrt{C}}{2}r\left(\log \frac{1}{r}\right)^{-\beta /2}.
\end{equation}
From the proof of Lemma 5.2 in \cite{XiongZheng}, it is clear that if 
$$
h(t)=Ct\left(\log \frac{1}{t}\right)^{-\beta},
$$
then $g=h^{-1}$ satisfies $g(t) \le \frac{1}{C}t\left(\log \frac{1}{t}\right)^{\beta}$ for $t \ll 1$.
So (\ref{deltamore}) implies
\begin{equation}\label{rless}
r \le \frac{2}{\sqrt{C}}\delta_{\Omega}(z_0)\left( \log \frac{1}{\delta_{\Omega}(z_0)} \right)^{\beta /2}.
\end{equation}
By (\ref{deltaless}), (\ref{deltamore}), (\ref{rless}), and (\ref{bergmankernelring}), we have
$$
K_A(z_0)\le I\cdot \frac{1}{\delta_{\Omega}(z_0)^2},
$$
where
\begin{eqnarray}
\notag I &=& \frac{1}{2\pi \beta \log\log \frac{1}{\delta_{\Omega}(z_0)} +2\pi \beta \log \left(1+(\log \frac{1}{\delta_{\Omega}(z_0)})^{-1}(\log \frac{\sqrt{C}}{2}-\frac{\beta}{2}\log \log \frac{1}{\delta_{\Omega}(z_0)})\right)-2\pi \log C}\\
\notag &&+\frac{32C}{3\pi} \left(  \log \frac{1}{\delta_{\Omega}(z_0)}-\frac{\beta}{2}\log \log \frac{1}{\delta_{\Omega}(z_0)}+\log \frac{\sqrt{C}}{2}   \right)^{-\beta} \label{Iless}\\
\notag&=& \frac{1}{2\pi \beta \log\log \frac{1}{\delta_{\Omega}(z_0)} +O\left(\frac{\log \log \frac{1}{\delta_{\Omega}(z_0)}}{\log \frac{1}{\delta_{\Omega}(z_0)}}\right)}+O\left(\left(\log \frac{1}{\delta_{\Omega}(z_0)}\right)^{-\beta}\right).
\end{eqnarray}
On the other hand, we have $K_A(z_0)\ge K_{\Omega}(z_0) \ge M\cdot \frac{1}{\delta_{\Omega}(z_0)^2\log \log \frac{1}{\delta_{\Omega}(z_0)}}$ for some $M>0$. Thus, for each $\beta,C$ and $r_0>0$, we can find a point $z_0 \in \Omega$ with $\delta_{\Omega}(z_0)<r_0$, such that
$$
0<M \le I \cdot \log \log \frac{1}{\delta_{\Omega}(z_0)}.
$$
Let $C$ be fixed and $r_0 \to 0$. Then we obtain
$$
0<M\le \frac{1}{2\pi \beta}, \,\,\forall \beta>0.
$$
But it is impossible for $\beta \gg 1$, which completes the proof of $(3)$.

\end{proof}

\section{Proof of Theorem 1.2}
We denote the Poincar\'e metric of a hyperbolic domain $\Omega \subset \mathbb{C}$ by $\rho_{\Omega}(z)|\mathrm{d}z|$, whose curvature is $-1$. It is obtained by Beardon-Pommerenke \cite{BeardonPommerenke} that
\begin{equation}\label{BeardonPommerenke}
\rho_{\Omega}(z) \asymp \frac{1}{\delta_{\Omega}(z)(\beta_{\Omega}(z)+C)},
\end{equation}
where $$\beta_{\Omega}(z)=\inf \left\{\left|\log \left|\frac{z-a}{b-a}\right|\right|;a\in \partial \Omega,b\in \partial \Omega, |z-a|=\delta_{\Omega}(z)\right\}.$$ Let $\partial \Omega$ be $h$-uniformly perfect. For each $w \in \Omega$, choose a point $a\in \partial \Omega$ with $ |w-a|=\delta_{\Omega}(w)$. If $w$ is sufficiently close to $\partial \Omega$, then there exists a point
$$
b \in \{ z\in \mathbb{C}; h(\delta_{\Omega}(w)) \le |z-a| \le \delta_{\Omega}(w) \}\cap \partial \Omega.
$$
Then
$$
\beta_{\Omega}(w) \le \left| \log \frac{|w-a|}{|b-a|}  \right| \le \log \frac{\delta_{\Omega}(w)}{h(\delta_{\Omega}(w))}.
$$
In view of (\ref{BeardonPommerenke}), we have
$$
\rho_{\Omega}(z) \gtrsim \frac{1}{\delta_{\Omega}(z)\log \frac{\delta_{\Omega}(z)}{h(\delta_{\Omega}(z))}}.
$$
In particular, if $\partial \Omega$ satisfies $(U)_{1,\alpha}$ or $ (U)_{2,\beta}$ , then we obtain
$$
\rho_{\Omega}(z) \gtrsim \frac{1}{\delta_{\Omega}(z) \log \frac{1}{\delta_{\Omega}(z)}}
$$
or
$$
\rho_{\Omega}(z) \gtrsim \frac{1}{\delta_{\Omega}(z) \log \log \frac{1}{\delta_{\Omega}(z)}},
$$
respectively. The first estimate is trivial. 

The proof of Theorem 1.2 is almost identical to that of Theorem 1.1, except that the Bergman kernel is replaced by Poincar\'e metric. For completeness, we provide the details here. Recall that 
for $R=\left\{z\in \mathbb{C}; \frac{1}{R}<|z|<R\right\}$,  we have 
\begin{equation}\label{3.2'}
\rho_R(z_0)=\frac{\pi}{2\log R}\cdot \frac{1}{|z|\cos\left( \frac{\pi \log |z|}{2\log R} \right)},
\end{equation}
and for $R=\{z\in \mathbb{C}; re^{-m}<|z|<re^m\}$,  we have 
\begin{equation}\label{3.2}
\rho_R(z_0)=\frac{\pi}{2 rm},
\end{equation} 
where $z_0 \in R$, $|z_0|=r$  (cf. \cite{BeardonMinda, BeardonPommerenke}). 

\begin{proof}[Proof of Theorem 1.2]
$(1)$ Let $t \in (0,\frac{1}{2}]$. Suppose on the contrary that  $\partial \Omega$ does not satisfies condition $(U)_{1,\alpha}$ for every $\alpha>1$. Then, for every $C,r_0>0,$ there exist a point $a \in \partial \Omega$ and a radius $r \in (0,r_0)$ such that
$$
\{ z \in \mathbb{C}; Cr^{\alpha} \le |z-a| \le r   \} \cap \partial \Omega =\emptyset.
$$
Let $A:=\{ z \in \mathbb{C}; Cr^{\alpha} < |z-a| < r   \}$. If $r_0<\mathrm{diam}( \Omega)/2$, then $\overline{A} \subset \Omega$. Take $z_0 \in A$ with $|z_0-a|=C^tr^{t(\alpha-1)+1}$. Let
$$ 
T: A \to R:=\left\{z\in \mathbb{C}; C^{\frac{1}{2}}r^{\frac{\alpha-1}{2}}<|z|<C^{-\frac{1}{2}}r^{\frac{1-\alpha}{2}} \right\}, z \to C^{-\frac{1}{2}}r^{-\frac{\alpha+1}{2}}(z-a)
$$
and $z_1=T(z_0)$.
Then $|z_1|=C^{t-\frac{1}{2}}r^{t(\alpha-1)+\frac{1-\alpha}{2}}$.
Hence, (\ref{3.2'}) implies
\begin{eqnarray*}
\rho_A(z_0)&=&\rho_R(z_1)|T'(z_0)|\\
&=&  \frac{\pi C^{-\frac{1}{2}}r^{-\frac{\alpha+1}{2}}}{2\log(C^{-\frac{1}{2}}r^{\frac{1-\alpha}{2}})}\cdot \left(C^{-\frac{1}{2}}r^{-\frac{\alpha+1}{2}}|z_0-a|\cos\left(\frac{\pi \log(C^{t-\frac{1}{2}}r^{t(\alpha-1)+\frac{1-\alpha}{2}})}{2\log(C^{-\frac{1}{2}}r^{\frac{1-\alpha}{2}})} \right)\right)^{-1}\\
&\le&\frac{\pi}{2\log(C^{-\frac{1}{2}}r^{\frac{1-\alpha}{2}})\delta_{\Omega}(z_0)}\cdot \frac{1}{\cos\left( \frac{\pi}{2}(1-2t)\right)}.
\end{eqnarray*}
We infer from (\ref{restimate}) that
\begin{eqnarray*}
\rho_A(z_0)
&\le& \frac{ (t(\alpha-1)+1)\pi}{(\alpha-1) \delta_{\Omega}(z_0)}\cdot \frac{1}{\log \frac{1}{\delta_{\Omega}(z_0)}+\log \frac{C^t}{2}-\frac{t(\alpha-1)+1}{(\alpha-1)}\log C}\cdot \frac{1}{\cos\left( \frac{\pi}{2}(1-2t)\right)}.
\end{eqnarray*}
Notice that $\rho_A(z_0) \ge \rho_{\Omega}(z_0) \ge \frac{M}{\delta_{\Omega}(z_0)\log \frac{1}{\delta_{\Omega}(z_0)}}$, since $A \subset \Omega$. Thus, for each $\alpha>1, t\in (0,\frac{1}{2}]$, and $C,r_0>0$, there exist $r\in (0,r_0)$ and $z_0 \in \Omega$ with $\delta_{\Omega}(z_0)<r_0$, such that
\begin{eqnarray*}
1<M &\le& \frac{(t(\alpha -1)+1)\pi}{(\alpha-1)} \cdot \frac{\log \frac{1}{\delta_{\Omega}(z_0)}}{\log \frac{1}{\delta_{\Omega}(z_0)}+\mathrm{const}(t,C,\alpha)} \cdot  \frac{1}{\cos\left( \frac{\pi}{2}\cdot (1-2t)\right)}.
\end{eqnarray*}
Letting $r_0 \to 0$, we obtain
$$
1<M \le \frac{(t(\alpha-1) +1)\pi}{(\alpha -1)}\cdot\frac{1}{\cos (\frac{\pi}{2}(1-2t))}, \,\,\,\forall \alpha>1, t\in (0,1/2].
$$
Then letting $\alpha \to \infty$, we conclude
$$
1 <M \le \frac{\pi t}{\sin (\pi t)}, \,\,\,\forall t \in (0,1/2].
$$
But it is impossible if $t$ is sufficiently small, which implies $(1)$. 

$(2)$ Suppose that for every $\beta>0$, $\partial \Omega$ does not satisfy condition $(U)_{2,\beta}$. Then, for every $C,r_0>0,$ there exist a point $a\in \partial \Omega$ and a radius $r\in (0,r_0)$ such that 
$$
\left\{z\in \mathbb{C}; Cr\left(\log \frac{1}{r}\right)^{-\beta} \le |z-a|\le r  \right\} \cap \partial \Omega =\emptyset.
$$
Let $A:= \left\{z\in \mathbb{C}; Cr\left(\log \frac{1}{r}\right)^{-\beta} < |z-a|< r  \right\}$. For $r_0< \mathrm{diam}(\Omega)/2$, we have $\overline{A} \subset \Omega$. Take $z_0 \in A$ with $|z_0-a|=\sqrt{C}r\left(\log \frac{1}{r}\right)^{-\beta/2}$. Then
$$
A=\left\{ z\in \mathbb{C}; |z_0-a|e^{-m}  <|z-a|<  |z_0-a| e^m   \right\},
$$
where
$$
m=-\log \left( \sqrt{C}\left( \log \frac{1}{r} \right)^{-\beta/2}\right).
$$
Then, (\ref{3.2}) implies
$$
\rho_A(z_0) =\frac{\pi}{2|z_0-a| m}\le \frac{\pi }{2\delta_{\Omega}(z_0)}\cdot \frac{1}{\frac{\beta}{2}\log \log \frac{1}{r}+\log \frac{1}{\sqrt{C}}}.
$$
In view of (\ref{rless}), we obtain
\begin{eqnarray*}
\rho_A(z_0) &\le& \frac{\pi}{2\delta_{\Omega}(z_0)}\cdot \frac{1}{\frac{\beta}{2}\log \log \left(\frac{1}{\delta_{\Omega}(z_0)} \cdot \frac{\sqrt{C}}{2} \cdot \left(\log \frac{1}{\delta_{\Omega}(z_0)}\right)^{-\beta}\right)+\log \frac{1}{\sqrt{C}}}\\
&\le& \frac{\pi}{2\delta_{\Omega}(z_0)}\cdot \frac{1}{\frac{\beta}{2}\log \log \frac{1}{\delta_{\Omega}(z_0)}+O\left (\frac{\frac{\beta}{2}\log\log \frac{1}{\delta_{\Omega}(z_0)}}{\log \frac{1}{\delta_{\Omega}(z_0)}}\right) }.
\end{eqnarray*}
Notice that $\rho_A(z_0) \ge \rho_{\Omega}(z_0) \ge \frac{M}{\delta_{\Omega}(z_0)\log \log \frac{1}{\delta_{\Omega}(z_0)}}$. Thus, for every $\beta>0$ and $C,r_0>0,$ there exists $z_0 \in \Omega$ with $\delta_{\Omega}(z_0)<r_0$, such that
$$
0<M \le \frac{\pi}{2} \cdot \frac{\log \log \frac{1}{\delta_{\Omega}(z_0)}}{\frac{\beta}{2}\log \log \frac{1}{\delta_{\Omega}(z_0)}+O\left (\frac{\frac{\beta}{2}\log\log \frac{1}{\delta_{\Omega}(z_0)}}{\log \frac{1}{\delta_{\Omega}(z_0)}}\right) }.
$$
Letting $r_0 \to 0$, we obtain
$$
0<M \le \frac{\pi}{\beta},\,\,\,\forall \beta >0.
$$
But it is impossible if $\beta \gg 1$, which completes the proof of $(2)$.
\end{proof}

\section{Proof of Theorem 1.3}
In \cite{Sugawa2001}, Sugawa introduced the concept of the uniform $LHMD$ condition, which provides a quantitative characterization of the boundary regularity of the Dirichlet problem. This condition is also equivalent to the uniform $\Delta$-regularity. Let $\Omega$ be a domain in $\mathbb{C}$. For $a\in \partial \Omega$ and $r>0$, $w_{a,r}(z)$ denotes the harmonic measure of $\partial D(a,r)\cap \Omega$ relative to $D(a,r)\cap \Omega$. In this paper, we say that a domain $\Omega \subset \mathbb{C}$ satisfies the property $(LHMD)_{1,\gamma}$ or $(LHMD)_{2,\eta}$ if there exist constant $C,r_0>0$ such that 
$$
w_{a,r}(z)\le C\left( \frac{\log \frac{1}{|z-a|}}{\log \frac{1}{r}} \right)^{-\gamma}
$$
or
$$
w_{a,r}(z) \le C \exp{\left(-\eta \left(\frac{\log \frac{1}{|z-a|}}{\log \log \frac{1}{|z-a|}}-\frac{\log \frac{1}{r}}{\log \log \frac{1}{r}}  \right)   \right)}
$$
for every $a \in \partial \Omega, r \in (0,r_0)$, and $z\in D(x,r)\cap \Omega$, respectively.

\begin{prop}
Let $\Omega$ be a domain in $\mathbb{C}$.

$(1)$ If $\Omega$ satisfies $(LHMD)_{1,\gamma}$ for some $\gamma>0$, then there exists $\alpha>1$ such that $\Omega$ satisfies $(\Delta)_{1,\alpha}$.

$(2)$ If $\Omega$ satisfies $(LHMD)_{2,\eta}$ for some $\eta>0$, then there exists $\beta>0$ such that $\Omega$ satisfies $(\Delta)_{2,\beta}$.

\end{prop}
\begin{proof}
$(1)$The condition $(LHMD)_{1,\gamma}$ implies that there exist constant $C,r_0>0$ such that for every $a \in \partial \Omega$ and $r\in (0,r_0)$, we have
$$
w_{a,r}(z) \le C \left(\frac{\log \frac{1}{|z-a|}}{\log \frac{1}{r}}  \right)^{-\gamma}, \,\,\,\forall z \in \Omega \cap D(a,r).
$$
Let $h(r)=r^{\alpha}$ for some $\alpha>1$. Then we have
\begin{equation}
w_{a,r}(z) \le C\left( \frac{\log \frac{1}{r^{\alpha}}}{\log \frac{1}{r}}  \right)^{-\gamma}=C\alpha^{-\gamma}, \,\,\,\forall z \in \Omega \cap \partial D(a,h(r)). \label{prop 4.1.1}
\end{equation}
For $\alpha \gg 1$, let $\varepsilon=1-C\alpha^{-\gamma} \in (0,1)$. Then $(\ref{prop 4.1.1})$ implies the condition $(\Delta)_{1,\alpha}$.

$(2)$ The condition $(LHMD)_{2,\eta}$ implies that there exist some constant $C,r_0>0$ such that for every $a \in \partial \Omega$ and $r\in (0,r_0)$ we have
$$
w_{a,r}(z) \le C \exp{\left[-\eta\left( \frac{\log \frac{1}{|z-a|}}{\log \log \frac{1}{|z-a|}}-\frac{\log \frac{1}{r}}{\log \log \frac{1}{r}}    \right)  \right]}, \,\,\,\forall z \in \Omega \cap D(a,r).
$$
Let $h(r)=r\left( \log \frac{1}{r} \right)^{\beta}$ for some $\beta>0$, then we have
\begin{equation}
w_{a,r}(z) \le C\exp{\left[ \eta \left( \frac{\log r-\beta \log \log \frac{1}{r}}{\log \left(\log \frac{1}{r}+\beta \log\log \frac{1}{r}  \right)} -\frac{\log r}{\log \log \frac{1}{r}} \right)   \right]}, \,\,\,\forall z \in \Omega \cap \partial D(a,h(r)). \label{prop 4.1.2}
\end{equation}
Let $I(r):=\left( \frac{\log r-\beta \log \log \frac{1}{r}}{\log \left(\log \frac{1}{r}+\beta \log\log \frac{1}{r}  \right)} -\frac{\log r}{\log \log \frac{1}{r}} \right) $. Then
\begin{eqnarray*}
I&=&\frac{\log r}{\log \log \frac{1}{r}+O\left(\frac{\beta \log \log \frac{1}{r}}{\log \frac{1}{r}}\right)} -\frac{\beta \log\log \frac{1}{r}}{\log \log \frac{1}{r}+O\left(\frac{\beta \log \log \frac{1}{r}}{\log \frac{1}{r}}\right)} -\frac{\log r}{\log \log \frac{1}{r}}\\
&=& \log r \left[ \frac{O\left(\frac{\beta \log \log \frac{1}{r}}{\log \frac{1}{r}}\right)}{\log \log \frac{1}{r}\left( \log\log \frac{1}{r}+O\left(\frac{\beta \log \log \frac{1}{r}}{\log \frac{1}{r}}\right)  \right)}  \right]-\frac{\beta}{1+O\left(\frac{\beta}{\log \frac{1}{r}} \right)}\\
&\rightarrow&-\beta, \,\,\,(r \to 0).
\end{eqnarray*}
Notice that for $\beta \gg 1$, $Ce^{-\frac{\eta \beta}{2}}<1$. Thus, there exists a constant $r_1 \in (0,r_0)$ such that $$Ce^{\eta I(r)}<Ce^{-\frac{\eta\beta}{2}}<1$$ for all $0<r<r_1$. Let $\varepsilon=1-Ce^{-\frac{\eta\beta}{2}}$. Then $\Omega$ satisfies the condition $(\Delta)_{2,\beta}$.

\end{proof}
\begin{prop}
Let $\Omega$ be a domain in $\mathbb{C}$.

$(1)$ If $\Omega$ satisfies $(\Delta)_{1,\alpha}$ for some $\alpha>1$, then there exists  $\alpha'>\alpha$ such that $\partial \Omega$ satisfies $(U)_{1,\alpha'}$.

$(2)$ If $\Omega$ satisfies $(\Delta)_{2,\beta}$ for some $\beta>0$, then there exists $\beta'>\beta$ such that $\partial \Omega$ satisfies $(U)_{2,\beta'}$.

\end{prop}
\begin{proof}
$(1)$ The condition $(\Delta)_{1,\alpha}$ means that there exist positive constant $\varepsilon, C$ and $r_0$ such that, for every $a \in \partial \Omega$ and $r\in (0,r_0)$, we have
$$
w_{a,r}(z) \le 1-\varepsilon,\,\,\,\forall z \in \Omega \cap \partial D(a,Cr^{\alpha}).
$$
We show $(1)$ by contradiction. Suppose on the contrary that for every $\alpha'>\alpha, C'>0$ and $r_1>0$, there exist $a \in \partial \Omega$ and $r\in (0,r_1)$ such that
$$
\{ z\in \mathbb{C}; C'r^{\alpha'}\le |z-a| \le r \} \cap \partial \Omega= \emptyset.
$$
Then for $r<r_1<\frac{\mathrm{diam}(\Omega)}{2}$ we have
$$
A:=\{ z\in \mathbb{C}; C'r^{\alpha'}< |z-a| < r \} \subset \Omega.
$$
Let $r<r_1 < \min \left\{ r_0, \frac{\mathrm{diam}(\Omega)}{2}, \left( \frac{C}{C'} \right)^{\frac{1}{\alpha'-\alpha}}  \right\}$. The function
$$
\phi(z):=\frac{\log \frac{|z-a|}{C'r^{\alpha'}}}{\log \frac{r}{C'r^{\alpha'}}}\\
$$
is harmonic on $A$, such that
\begin{eqnarray*}
\phi|_{\partial D(a,C'r^{\alpha'})}&=&0 \,\, \le \,\, w_{a,r}|_{\partial D(a,C'r^{\alpha'})}\\
\phi|_{\partial D(a,r)}&=&1 \,\,=\,\, w_{a,r}|_{\partial D(a,r)}.
\end{eqnarray*}
Thus, $\phi \le w_{a,r}$ on $A$. In particular, for $z\in \partial D(x,Cr^{\alpha})\subset A$, we have
$$
\phi(z)=\frac{\log \frac{Cr^{\alpha}}{C'r^{\alpha'}}}{\log \frac{r}{C'r^{\alpha'}}} \le w_{a,r}(z) \le 1-\varepsilon.
$$
That is, for every $\alpha'> \alpha, C'>0$ and $r_1 \ll 1$, there are $a \in \partial \Omega$ and $r \in (0,r_1)$ such that
$$
\frac{(\alpha'-\alpha)\log \frac{1}{r}+\log \frac{C}{C'}}{(\alpha'-1)\log \frac{1}{r}-\log C'}\le 1-\varepsilon.
$$
Let $C'=1$ and $r_1 \to 0$. Then we obtain
$$
\frac{\alpha'-\alpha}{\alpha'-1} \le 1-\varepsilon, \,\,\,\forall \alpha'>\alpha,
$$
which is impossible.

$(2)$ The condition $(\Delta)_{2,\beta}$ means that there exist some positive constants $\varepsilon, C$ and $r_0$, such that for every $a \in \partial \Omega$ and $r\in (0,r_0)$, we have
$$
w_{a,r}(z) \le 1-\varepsilon,\,\,\,\forall z \in \Omega \cap \partial D\left(a,Cr\left(\log \frac{1}{r}\right)^{-\beta}\right).
$$
Suppose that for every $\beta'>\beta, C'>0$ and $r_1>0$, there exist $a \in \partial \Omega$ and $r\in (0,r_1)$, such that
$$
\left\{ z\in \mathbb{C}; C'r\left( \log \frac{1}{r} \right)^{-\beta'}\le |z-a| \le r \right\} \cap \partial \Omega= \emptyset.
$$
Then for $r<r_1<\frac{\mathrm{diam}(\Omega)}{2}$, we have
$$
A:=\left\{ z\in \mathbb{C}; C'r\left(\log \frac{1}{r} \right)^{-\beta'}< |z-a| < r \right\} \subset \Omega.
$$
Consider the function
$$
\phi(z):=\frac{\log \frac{|z-a|}{C'r\left(\log \frac{1}{r}\right)^{-\beta'}}}{\log \frac{r}{C'r\left( \log \frac{1}{r}\right)^{-\beta'}}}.\\
$$
As in $(1)$ we have $\phi \le w_{a,r}$ on $A$. In particular, for $z\in \partial D\left(a,Cr\left( \log\frac{1}{r} \right)^{-\beta}\right)\subset A$, we have
$$
\phi (z)=\frac{(\beta'-\beta) \log \log \frac{1}{r}+\log \frac{C}{C'}}{\beta' \log \log \frac{1}{r}-\log C'}\le 1-\varepsilon.
$$
Let $C'=1$ and $r_1 \to 0$, we obtain
$$
\frac{\beta'-\beta}{\beta'} \le 1-\varepsilon, \,\,\,\forall \beta'>\beta,
$$
which leads to a contradiction as $\beta' \to +\infty$.

\end{proof}

To complete the proof of Theorem 1.3, we need the following result from \cite{ChenDensity}.
\begin{theorem}[Chen]
Let $\Omega$ be a bounded domain in $\mathbb{C}$ with $0 \in \partial \Omega$. Let $0 \le h \le 1$ be a harmonic function on $\Omega$ such that $h=0$ n.e. on $\partial \Omega \cap D(0,r_1)$ for some $r_1<1$. For all $\kappa \in \left(0,\frac{1}{16}\right)$ and $r \le \kappa r_1$, we have
$$
\sup_{\Omega \cap D(0,r)} \le \exp{\left[ -C_{\kappa}\int_r^{\kappa r_1}\left( t\log \frac{t/\kappa}{2\mathrm{Cap}(K_t(0))} \right)^{-1} \mathrm{d}t   \right]},
$$
where $K_t(0):=\overline{D(0,t)}-\Omega$.
\end{theorem}
\begin{prop}
Let $\Omega$ be a domain in $\mathbb{C}$.

$(1)$ If $\Omega$ satisfies $(U)_{1,\alpha}$ for some $\alpha \in (1,2)$, then there exists $\gamma>0$ such that $\Omega$ satisfies $(LHMD)_{1,\gamma}$.

$(2)$ If $\Omega$ satisfies $(U)_{2,\beta}$ for some $\beta>0$, then there exists $\eta>0$ such that $\Omega$ satisfies $(LHMD)_{2,\eta}$.

\end{prop}
\begin{proof}
$(1)$ Theorem 1.3 of \cite{XiongZheng} implies that there exist some constants $C, r_0>0$ such that
$$
\mathrm{Cap}(K_r(a)) \ge Cr^{\frac{1}{2-\alpha}}
$$
for every $a\in \partial \Omega$ and $r\in (0,r_0)$, where $K_r(a):=\overline{D(a,r)}-\Omega$. 

Let $\kappa \in \left(0,\frac{1}{16}\right)$ and $r_1 \in \left( 0, \frac{r_0}{\kappa} \right)$ be fixed. For every $a \in \partial \Omega$ and $r \in (0,r_1)$, $w_{a,r}$ is a harmonic function on $\Omega_r:=\Omega\cap D(a,r)$, and $w_{a,r}(z)=0$ n.e. on $\partial \Omega_r \cap D\left(0,\frac{r}{2}\right)$. Then Theorem 4.3 implies that for $z\in \Omega_r$ with $|z-a|<\frac{\kappa r}{2}$, we have
\begin{eqnarray*}
w_{a,r}(z) &\le& \exp{\left[ -C_{\kappa}\int^{\frac{\kappa r}{2}}_{|z-a|} \frac{1}{t\log \frac{t}{2\kappa \mathrm{Cap}(K_t(a))}} \mathrm{d}t  \right]} \\
&\le& \exp{\left[ -C_{\kappa}\int^{\frac{\kappa r}{2}}_{|z-a|} \frac{1}{t\log \frac{t}{2C\kappa t^{\frac{1}{2-\alpha}}}} \mathrm{d}t  \right]} \\
&=& \exp{\left[ -C_{\kappa}\int^{\frac{\kappa r}{2}}_{|z-a|} \frac{1}{t\left[\left( \frac{1}{2-\alpha}-1 \right)\log \frac{1}{t}+\log \frac{1}{2\kappa C} \right]} \mathrm{d}t  \right]} \\
&\le& \exp{\left[ -\gamma \int^{\frac{\kappa r}{2}}_{|z-a|} \frac{1}{t\log \frac{1}{t}} \mathrm{d}t  \right]} \\
&=& \exp{\left[ -\gamma \left( \log \log \frac{1}{|z-a|}-\log\log \frac{2}{\kappa r}      \right)    \right]}\\
&\le& C_1\exp{\left[ -\gamma \left( \log \log \frac{1}{|z-a|}-\log\log \frac{1}{r}      \right) \right]}\\
&=& C_1 \left(  \frac{\log \frac{1}{|z-a|}}{\log \frac{1}{r}}  \right)^{-\gamma},
\end{eqnarray*}
where $\gamma$ and $C_1$ depend on $\kappa, r_1, \alpha, C$. When $\frac{\kappa r}{2}\le |z-a| \le r$, we have $w_{a,r}\le 1$, and 
$$
\left(  \frac{\log \frac{1}{|z-a|}}{\log \frac{1}{r}}  \right)^{-\gamma} \ge C_2,
$$
where $C_2$ depends on $\kappa, r_1,\alpha$ and $C$. Thus, there exists a constant $C_3>0$ such that
\begin{equation}
w_{a,r}(z) \le  C_3 \left(  \frac{\log \frac{1}{|z-a|}}{\log \frac{1}{r}}  \right)^{-\gamma} \label{4.4.1}
\end{equation}
for every $z \in \Omega \cap D(a,r)$. The inequality $(\ref{4.4.1})$ holds for every $a\in \partial \Omega$ and $r \in (0,r_1)$, which implies $(1)$.

$(2)$ Theorem 1.3 of \cite{XiongZheng} implies that there exist constants $C, r_0>0$ such that
$$
\mathrm{Cap}(K_r(a)) \ge Cr \left( \log \frac{1}{r} \right)^{-\beta}
$$
for every $a\in \partial \Omega$ and $r\in (0,r_0)$. 

We use Theorem 4.3 again to deduce that for $z\in \Omega_r$ with $|z-a|<\frac{\kappa r}{2}$
\begin{eqnarray*}
w_{a,r}(z) &\le& \exp{\left[ -C_{\kappa}\int^{\frac{\kappa r}{2}}_{|z-a|} \frac{1}{t\log \frac{t}{2\kappa \mathrm{Cap}(K_t(a))}} \mathrm{d}t  \right]} \\
&\le& \exp{\left[ -C_{\kappa}\int^{\frac{\kappa r}{2}}_{|z-a|} \frac{1}{t\left( \beta \log\log \frac{1}{t}+\log \frac{1}{2\kappa C} \right)} \mathrm{d}t  \right]} \\
&\le& \exp{\left[ -\eta \int^{\frac{\kappa r}{2}}_{|z-a|} \frac{1}{t\log \log \frac{1}{t}} \mathrm{d}t  \right]}, \\
\end{eqnarray*}
where $\eta$ depends on $\kappa, r_1, \alpha$ and $C$. 
Notice that
$$
\left( \frac{\log t}{\log \log \frac{1}{t}}  \right)' =\frac{1}{t\log \log \frac{1}{t}} \left( 1- \frac{1}{\log \log \frac{1}{t}} \right) \le \frac{1}{t \log \log \frac{1}{t}}.
$$
Thus, 
\begin{eqnarray*}
w_{a,r}(z) &\le& \exp{ \left[ -\eta\int_{|z-a|}^{\frac{\kappa r}{2}}  \left( \frac{\log t}{\log \log \frac{1}{t}}  \right)'  \mathrm{d}t  \right]}\\
&=& \exp{\left[- \eta\left(  \frac{\log \frac{1}{|z-a|}}{\log \log \frac{1}{|z-a|}}   - \frac{\log \frac{2}{\kappa r}}{\log \log \frac{2}{\kappa r}}    \right)  \right]}\\
&\le& C_1\exp{\left[ -\eta\left(  \frac{\log \frac{1}{|z-a|}}{\log \log \frac{1}{|z-a|}}   - \frac{\log \frac{1}{r}}{\log \log \frac{1}{r}}    \right)  \right]}.
\end{eqnarray*}
When $\frac{\kappa r}{2}\le |z-a| \le r$, $w_{a,r}\le 1$, and 
$$
\exp{\left[ -\eta\left(  \frac{\log \frac{1}{|z-a|}}{\log \log \frac{1}{|z-a|}}   - \frac{\log \frac{1}{r}}{\log \log \frac{1}{r}}    \right)  \right]} \ge C_2,
$$
where $C_2$ depends on $\kappa, r_1, \alpha, C$. Thus, there exists a constant $C_3>0$ such that
\begin{equation}
w_{a,r}(z) \le  C_3 \exp{\left[- \eta\left(  \frac{\log \frac{1}{|z-a|}}{\log \log \frac{1}{|z-a|}}   - \frac{\log \frac{1}{r}}{\log \log \frac{1}{r}}    \right)  \right]} \label{4.4.2}
\end{equation}
for every $z \in \Omega \cap D(a,r)$. The inequality $(\ref{4.4.2})$ holds for every $a\in \partial \Omega$ and $r \in (0,r_1)$, which implies $(2)$.

\end{proof}
Theorem 1.3 is obtained by combining Proposition 4.1, Proposition 4.2 and Proposition 4.4.

\section{Proof of Theorem 1.4}
Let 
$
g:(0,+\infty) \to (0,+\infty)
$
be an increasing function such that $\lim_{t \to 0} g(t)=0$. Let $E$ be a closed set in $\mathbb{C}$. Recall the concepts of Hausdorff content and Hausdorff measure: 
\begin{eqnarray*}
&\Lambda_{g}(E):= \inf \left\{ \sum_k g(\mathrm{diam} (B_k)); \{ B_k \} \text{is a countable covering of $E$ composed of closed discs} \right\},\\
&\mathcal{H}_g^{\varepsilon}(E):=\inf \left\{ \sum_k g(\mathrm{diam} (B_k)); \{ B_k \} \text{is a countable covering of $E$ such that $diam (B_k)<\varepsilon$} \right\},\\
&\mathcal{H}_g(E):=\lim_{\varepsilon \to 0} \mathcal{H}_g^{\varepsilon}(E).
\end{eqnarray*}
It is known that $\Lambda_g(E) \le \mathcal{H}_g(E)$, and $\Lambda_g(E)=0 \Leftrightarrow \mathcal{H}_g(E)=0$ (cf. \cite{Reshetnyak}). If $g(t)=t^{\alpha}, \alpha>0$, then we write $\Lambda^{\alpha}(E):=\Lambda_{g}(E)$, $\mathcal{H}^{\alpha}(E):= \mathcal{H}_g(E)$. The Hausdorff dimension of $E$ is $\dim_H(E):= \inf \{ \alpha>0; \mathcal{H}^{\alpha}(E)=0  \}.$

Now we focus on the proof of Theorem 1.4, which is inspired by \cite{Sugawa1998}. For a closed disc $B=\overline{B}(a,r)$ we write $\mathrm{rad}(B):=r$ and $\mathrm{cent}(B)=a$. 
\begin{lemma}
Let $E \subset \mathbb{C}$ be  $h$-uniformly perfect. Then there exists $r_0>0$ such that for every $a \in E, r \in (0,r_0)$, and $ \tilde{c}\in (0,\frac{1}{2})$,  there are closed discs $B_1$ and $B_2$, such that

$(1)$ $B_i \subset B:=\overline{B}(a,r), i=1,2$;

$(2)$ $B_1 \cap B_2=\emptyset$;

$(3)$ $ \mathrm{cent}(B_i) \in E, i=1,2$;

$(4)$ $\mathrm{rad} (B_i)=\tilde{c} \cdot h\left(\frac{r}{2}\right)=:\tilde{h}(r),i=1,2$.
\end{lemma}
\begin{proof}
Let $B_1=\overline{B}\left(a,\tilde{c}h\left(\frac{r}{2}\right)\right) \subset B$. Since $E$ is $h$-uniformly perfect, there exists a constant $r_0>0$ such that
$$
\{z \in \mathbb{C}; h(r)\le |z-a|\le r\} \cap E \ne \emptyset
$$
for every $a\in E$ and $r\in (0,r_0)$. 
Take a point
$$
b\in \left\{z\in \mathbb{C}; h\left(\frac{r}{2}\right) \le |z-a|\le \frac{r}{2}   \right\}\cap E,
$$
and let $B_2:= \overline{B}\left(b,\tilde{c}h\left(\frac{r}{2}\right)\right)$.  Then $B_1\cap B_2=\emptyset$, since $\tilde{c}<\frac{1}{2}$. It is also easy to see that $B_2 \subset B$. 
\end{proof}
Applying Lemma 5.1 inductively, we select a sequence $\{B_{i_1,\cdots,i_k}\}_{k\in \mathbb{N},i_j \in I}, I=\{1,2\}$ such that:
\begin{eqnarray*}
&(1)& B_{i_1,\cdots,i_k} \subset B_{i_1,\cdots,i_{k-1}};\\
&(2)& B_{i_1,\cdots,i_{k-1},1}\cap B_{i_1,\cdots,i_{k-1},2}=\emptyset;\\
&(3)& \mathrm{cent}(B_{i_1,\cdots,i_k})\in E;\\
&(4)& \mathrm{rad}(B_{i_1,\cdots,i_k}) =\tilde{h}^{\circ (k)}(r).\\
\end{eqnarray*}
Now we construct a Cantor-type set $\mathcal{C}:= \bigcap_{k=0}^{\infty}\bigcup_{(i_1,\cdots i_k)\in I^k} B_{i_1,\cdots,i_k}$. 

Let $f$ be a map defined by
$$
f: I^{\mathbb{N}} \to \mathcal{C}, \{i_k\}_{k \in \mathbb{N}} \to \bigcap_{k=1}^{\infty}B_{i_1,\cdots,i_k}.
$$
We equip $I$ and $I^{\mathbb{N}}$ the discrete and product topologies, respectively. Obviously, $\mathcal{C} \subset \mathbb{C}$ is Hausdorff, and Tychonoff Theorem implies that $I^{\mathbb{N}}$ is compact. It is also easy to verify that $f$ is a continuous bijection, so $f$ is a homeomorphism.  
We know that $I^{\mathbb{N}}$ carries a standard Bernoulli measure $\nu$ such that $\nu([i_1,\cdots,i_k])=2^{-k}$, for each $[i_1,\cdots,i_k]=\{ (j_l)_{l\in \mathbb{N}}; j_1=i_1, \cdots, j_k=i_k  \}$. Then $\mu:=f_{*}\nu$ is a probability measure on $\mathcal{C}$ such that

$(1)$ $supp \,\,\mu =\mathcal{C}$;

$(2)$ $\mu(B_{i_1,\cdots,i_k})=2^{-k}, \,\,\,\forall k \in \mathbb{N},(i_1,\cdots,i_k)\in I^k$.

\begin{proof}[Proof of Theorem 1.4]
$(1)$  $(\Rightarrow)$: Write $h(r)=C_0h_{1,\alpha}(r)=C_0r^{\alpha}$. Then there exists a constant $r_0>0$ such that for every $a\in E$ and $r\in (0,r_0)$, we can construct a Cantor-type set $\mathcal{C} \subset E \cap \overline{B}(a,r)$ as above. Let $A=\overline{B}(x,\rho)$.

If $\rho <r $, take $k \in \mathbb{N}$ such that $\tilde{h}^{\circ (k+1)}(r)<\rho \le \tilde{h}^{\circ k}(r)$. Let $$J:= \{ (i_1,\cdots,i_k) \in I^k; B_{i_1,\cdots,i_k}\cap A \ne \emptyset  \}.$$ For $(i_1,\cdots,i_k) \in J$, we have
$$
B_{i_1,\cdots,i_k} \subset \overline{B}(x,\rho+2\tilde{h}^{\circ k}(r)).
$$
Notice that these discs do not inersect each other. Therefore, we observe that
$$
\# J \cdot \pi \cdot (\tilde{h}^{\circ k}(r))^2 \le \pi(\rho +2\tilde{h}^{\circ k})^2.
$$
So
$$
\# J \le \left( \frac{\rho}{\tilde{h}^{\circ k}(r)}+2  \right)^2 \le 9.
$$
Thus, we have
\begin{equation}\label{muAless}
\mu (A) \le \mu \left(\bigcup_J B_{i_1,\cdots,i_k}\right)\le \sum_{J} \mu(B_{i_1,\cdots,i_k})\le 9\cdot 2^{-k}.
\end{equation}
Since $h(r)=h_{1,\alpha}(r)=C_0r^{\alpha}$ for some $C_0>0,\alpha>1$, we have $\tilde{h}(r)=\tilde{c}\cdot C_0\cdot 2^{-\alpha}r^{\alpha}=:C_1 r^{\alpha}$. Simple calculations indicate that
$$
\tilde{h}^{\circ k}(r) = C_2^{-1}(C_2 r)^{\alpha^k},
$$
where $C_2=C_1^{\frac{1}{\alpha-1}}$. Thus, for
$$
\tilde{h}^{\circ (k+1)}(r) <\rho \le \tilde{h}^{\circ k}(r),
$$
we have
$$
(C_2r)^{\alpha^{k+1}}<C_2\rho \le (C_2 r)^{\alpha^k}.
$$
If $r_0<\frac{1}{C_2}$, then
$$
k+1>\frac{1}{\log \alpha}\left( \log \log \frac{1}{C_2\rho}-\log\log \frac{1}{C_2r} \right) \ge k.
$$
We infer from (\ref{muAless}) that
\begin{eqnarray}
\notag \mu (A)&\le& 9\cdot 2^{-k}\\
\notag&<&9\cdot 2^{1-\frac{1}{\log \alpha}\left( \log \log \frac{1}{C_2\rho}-\log\log \frac{1}{C_2r} \right)}\\
\notag&=&18 \exp\left(\frac{\log 2}{\log \alpha}\log \log \frac{1}{C_2r} \right)\exp\left( -\frac{\log 2}{\log \alpha}\log \log \frac{1}{C_2\rho} \right)\\
&=& 18\left( \log \frac{1}{C_2r} \right)^{\frac{\log 2}{\log \alpha}}\left( \log \frac{1}{C_2\rho} \right)^{-\frac{\log 2}{\log \alpha}}. \label{muAless2}
\end{eqnarray}
Write $\gamma :=\frac{\log 2}{\log \alpha}.$ Let
$$
g_{1,\gamma}(t):= \left( \log \frac{2}{C_2t} \right)^{-\gamma}, t \in (0,2r_0).
$$
Here $r_0>0$ is sufficiently small such that $g_{1,\gamma}$ is increasing. We can extend the definition of $g_{1,\gamma}$ to $(0,+\infty)$ so that $g_{1,\gamma}$ still increases. Then $(\ref{muAless2})$ can be written as
\begin{equation}\label{muAless3}
\mu(A) \le 18 g_{1,\gamma}(2r)^{-1}\cdot g_{1,\gamma}(2\rho).
\end{equation}

When $\rho \ge r$, we still have
$$
\mu (A) \le 1 \le 18 g_{1,\gamma}(2r)^{-1}g_{1,\gamma}(2r) \le 18 g_{1,\gamma}(2r)^{-1}g_{1,\gamma}(2\rho),
$$
due to the monotonicity of $g_{1,\gamma}$.

Now for any countable cover $\{A_j\}$ of $\mathcal{C}$ by closed discs $A_j$ with radii $\rho_j$, we have
$$
1=\mu (\mathcal{C}) \le \sum_j \mu (A_j) \le \sum_j 18 g_{1,\gamma}(2r)^{-1}g_{1,\gamma}(2\rho_j).
$$
Thus,
$$
\sum_j g_{1,\gamma}(2\rho_j) \ge \frac{1}{18} g_{1,\gamma}(2r).
$$
Since $\{A_j\}$ is arbitrary, we obtain 
$$
\Lambda_{g_{1,\gamma}}(\mathcal{C}) \ge \frac{1}{18}g_{1,\gamma}(2r).
$$
Notice that $\mathcal{C} \subset \overline{B}(a,r)$, so
$$
\Lambda_{g_{1,\gamma}}(E \cap \overline{B}(a,r)) \ge \frac{1}{18}g_{1,\gamma}(2r).
$$
Since $a\in E, r\in (0,r_0)$ are arbitrary, we conclude the necessity of $(1)$.

$(\Leftarrow)$: Now, for every $a \in E$ and $ r \in (0,r_0)$, we have
\begin{equation}\label{Lambdage}
\Lambda_{g_{1,\gamma}}(E \cap \overline{B}(a,r)) \ge A\cdot \left( \log \frac{1}{2Cr} \right)^{-\gamma}.
\end{equation}
Here $\{\overline{B}(a,r+\varepsilon)\}$ is a countable cover of $\overline{B}(a,r)$ for each $\varepsilon>0$, so we have
$$
\Lambda_{g_{1,\gamma}}(\overline{B}(a,r)) \le \left( \log \frac{1}{2Cr}  \right)^{-\gamma}.
$$
Let $h(t)=t^{\alpha}$, where $\alpha>1$ is undetermined. Then
\begin{equation}\label{Lambdale}
\Lambda_{g_{1,\gamma}}(\overline{B}(a,h(r))) \le \left( \log \frac{1}{2Cr^{\alpha}}  \right)^{-\gamma}.
\end{equation}

We claim that there exist $\alpha>1$ and $0<r_1 \ll 1$ such that
$$
\left( \log \frac{1}{2Cr^{\alpha}}  \right)^{-\gamma}< A\cdot \left( \log \frac{1}{2Cr} \right)^{-\gamma},\,\,\,\forall \, r\in (0,r_1).
$$ 
To establish this claim, it suffices to show that
\begin{equation}\label{Age}
A> \left( \frac{\log \frac{1}{r}+\log \frac{1}{2C}}{ \alpha \log \frac{1}{r}+\log \frac{1}{2C}}  \right)^{\gamma}, \,\,\,\forall \, r \in (0,r_1).
\end{equation}
Notice that when $r \to 0$, the right side of $(\ref{Age})$ converges to $\left(\frac{1}{\alpha}\right)^{\gamma}$. So $(\ref{Age})$ holds for $\alpha \gg 1$ and $r_1 \ll 1$.

By the claim, inequalities (\ref{Lambdage}) and (\ref{Lambdale}), for every $a \in E$ and $r \in (0,r_1)$, we have
$$
\Lambda_{g_{1,\gamma}}(\overline{B}(a,h(r))) < \Lambda_{g_{1,\gamma}}(E \cap \overline{B}(a,r)).
$$
In particular, $E \cap \overline{B}(a,r) -\overline{B}(a,h(r)) \ne \emptyset$. We conclude that $E$ satisfies condition $(U)_{1,\alpha}$. 

$(2)$ $(\Rightarrow)$ The proof is essentially the same as in $(1)$, except that the form of (\ref{muAless3}) may change due to the change in the choice of the function $h$. Now, we have $\tilde{h}(r)=C_0r \left(\log \frac{2}{r}\right)^{-\beta}$ for some constant $C_0>0$. Write $s_0=r,s_{k+1}=\tilde{h}(s_k)$. In view of the derivation process of (3.25) in \cite{XiongZheng}, there exists a constant $C_1>0$ such that
\begin{equation}\label{ktimes}
\frac{1}{C_1}\beta k \le \frac{\log \frac{1}{s_k}}{\log \log \frac{2}{s_k}} -\frac{\log \frac{1}{r}}{\log \log \frac{2}{r}} \le C_1\beta k.
\end{equation}
When $\rho \in (s_{k+1},s_k]$, we see from the monotonicity of $\frac{\log \frac{1}{t}}{\log \log \frac{2}{t}}$ that
$$
C_1\beta(k+1) \ge \frac{\log \frac{1}{s_{k+1}}}{\log \log \frac{2}{s_{k+1}}} -\frac{\log \frac{1}{r}}{\log \log \frac{2}{r}} > \frac{\log \frac{1}{\rho}}{\log \log \frac{2}{\rho}} -\frac{\log \frac{1}{r}}{\log \log \frac{2}{r}}.
$$
Then $$k > \frac{1}{C_1 \beta} \left(\frac{\log \frac{1}{\rho}}{\log \log \frac{2}{\rho}} -\frac{\log \frac{1}{r}}{\log \log \frac{2}{r}} \right) -1.$$ 
In view of (\ref{muAless}), we obtain
\begin{equation}\label{muAless4}
\mu(A) \le 9 \cdot 2^{-k} \le 18 g_{2,\eta}(2r)^{-1}g_{2,\eta}(2\rho),
\end{equation}
where $\eta =\frac{\log 2}{C_1\beta}$ and $g_{2,\eta}$ is a function increasing on $(0,+\infty)$, such that $$g_{2,\eta}(t)=\exp\left(-\eta \frac{\log \frac{2}{t}}{\log \log \frac{4}{t}}\right)$$
on $(0,2r_0)$.
Replacing   (\ref{muAless3}) in $(1)$ with (\ref{muAless4}) to obtain the necessity part of $(2)$.

$(\Leftarrow)$ Similar to $(1)$, we have
$$
\exp \left( -\eta \frac{\log \frac{2}{r}}{\log \log \frac{4}{r}} \right) \ge \Lambda_{g_{2,\eta}} (\overline{B}(a,r))\ge \Lambda_{g_{2,\eta}} (E \cap \overline{B}(a,r)) \ge A \cdot \exp \left( -\eta \frac{\log \frac{2}{r}}{\log \log \frac{4}{r}} \right).
$$
Let $h(r)=r\left( \log \frac{1}{r} \right)^{-\beta}$, then
$$
\Lambda_{g_{2,\eta} }(\overline{B}(a,h(r))) \le \exp \left( -\eta \cdot \frac{\log \frac{2}{r\left(\log \frac{1}{r}\right)^{-\beta}}}{\log \log \frac{4}{r\left(\log \frac{1}{r} \right)^{-\beta}}} \right).
$$

We claim that there exist constants $\beta>0$ and $0<r_2 \ll 1$, such that
$$
\exp \left( -\eta \cdot \frac{\log \frac{2}{r\left(\log \frac{1}{r}\right)^{-\beta}}}{\log \log \frac{4}{r\left(\log \frac{1}{r} \right)^{-\beta}}} \right) <A \cdot \exp \left( -\eta \frac{\log \frac{2}{r}}{\log \log \frac{4}{r}} \right) , \,\,\,\forall r \in (0,r_2).
$$
In fact, it suffices to show that
$$
\frac{\log A}{\eta} > \frac{\log \frac{2}{r}}{\log \left( \log \frac{4}{r}  \right)}-\frac{\log \frac{2}{r}+\beta\log \log \frac{1}{r}}{\log \left(\log \frac{4}{r} +\beta \log \log\frac{1}{r}\right)}=:I,\,\,\,\forall \,r \in (0,r_2).
$$
Here
\begin{eqnarray*}
I&=& \frac{\log \frac{2}{r}}{\log \left( \log \frac{4}{r}  \right)}-\frac{\log \frac{2}{r}+\beta\log \log \frac{1}{r}}{\log\log \frac{4}{r} +\log \left( 1+\frac{\beta \log \log\frac{1}{r}}{\log \frac{4}{r}}\right)}\\
&=& \frac{\log \frac{2}{r}}{\log \left( \log \frac{4}{r}  \right)}-\frac{\log \frac{2}{r}}{\log \left( \log \frac{4}{r} \right) +O\left(\frac{\beta \log \log\frac{1}{r}}{\log \frac{4}{r}}  \right)}-\beta \frac{\log \log \frac{1}{r}}{\log \left(\log \frac{4}{r} +\beta \log \log\frac{1}{r}\right)}\\
&=&\frac{\log \frac{2}{r}}{\log \left( \log \frac{4}{r}  \right)} \left( \frac{O\left(\frac{\beta \log \log\frac{1}{r}}{\log \frac{2}{r}} \right)}{\log \log \frac{2}{r}+O\left(\frac{\beta \log \log\frac{1}{r}}{\log \frac{4}{r}} \right)  } \right)  -\beta \frac{\log \log \frac{1}{r}}{\log \left(\log \frac{4}{r} +\beta \log \log\frac{1}{r}\right)}  \\
&\to & -\beta, \,\,\,(r \to 0).
\end{eqnarray*}
So the claim is verified if we choose $\beta \gg 1$ and $r_2 \ll 1$. The proof is concluded as in $(1)$.
\end{proof}

\section*{Appendix}
\begin{example}
There exists a compact set $E \subset \mathbb{C}$ such that $\dim_H(E)=0$ and $E$ satisfies condition $(U)_{1,\alpha}$ (or condition $(U)_{2,\beta}$).
\end{example}
Consider a Cantor-type set $\mathcal{C}$: choose a sequence $\{ l_j\}_{j=0}^{\infty}$, where $l_j>0$ and $l_{j+1}<\frac{l_j}{2}$. Set $\mathcal{C}_0=[0,l_0]$ and define $\mathcal{C}_j$ to be a union of $2^j$ closed intervals inductively, such that $\mathcal{C_j}$ is obtained by removing from the middle of each inteval in $\mathcal{C}_{j-1}$ an open subinteval of length $l_{j-1}-2l_j$. 
For example, $\mathcal C_1=[0,l_1]\cup[l_0-l_1,l_0]$, $\mathcal C_2=[0,l_2]\cup[l_1-l_2,l_1]\cup[l_0-l_1,l_0-l_1+l_2]\cup[l_0-l_2,l_0]$, etc. Write
\[
\mathcal C_j=\bigcup^{2^j}_{k=1}I_{j,k},
\]
where each $I_{j,k}$ is a closed inteval of length $l_j$, lying on the left of $I_{j,k+1}$. We set
\[
\mathcal C:=\bigcap^\infty_{j=0} \mathcal C_j.
\]
For every $\gamma>0$ and $ j \in \mathbb{N}$, we have
$$
\Lambda^{\gamma}(\mathcal{C}) \le \Lambda^{\gamma}(\mathcal{C}_j) \le \sum_{k=1}^{2^j}l_j^{\gamma}=2^j\cdot l_j^{\gamma}.
$$

Let $l_j=l_0^{\alpha^j}$, then 
$$0 \le \Lambda^{\gamma}(\mathcal{C}) \le \lim_{j \to \infty} \Lambda^{\gamma}(\mathcal{C}_j) =0$$ 
for each $\gamma>0$, which implies $\dim_H(\mathcal{C})=0.$ Also, this Cantor-type set $\mathcal{C}$ satisfies condition $(U)_{1,\alpha}$, see \cite{XiongZheng}.

Similarly, if we take $l_j=l_{j-1}\left( \log \frac{1}{l_j} \right)^{-\beta}<\frac{l_j}{2}$, then the corresponding $\mathcal{C}$ satisfies condition $(U)_{2,\beta}$ and $dim_{H}(\mathcal{C})=0$.

{\bf Acknowledgements.} We are grateful to Prof. Bo-Yong Chen and Dr. Yuan-pu Xiong for many inspiring discussions and critical suggestions.

\end{document}